\documentclass[amsfonts]
{article}

\usepackage{mathrsfs}
\usepackage{amsmath} 
\usepackage{amssymb }
\usepackage{amsthm}
\usepackage{url}
\usepackage{graphicx}		
\usepackage{epsfig,psfrag}  
\usepackage{enumerate}

\newtheorem{lemma}{Lemma}[section]
\newtheorem{prop}[lemma]{Proposition}
\newtheorem{theorem}{Theorem}
\newtheorem{cor}[lemma]{Corollary}

\theoremstyle{definition}

\newtheorem{rmrk}[lemma]{Remark}

\newcommand{\e}{{\mathbf e}}

\newcommand{\PGL}{\mathrm{PGL}}

\renewcommand{\r}{\mathbf{r}}

\newcommand{\cO}{\mathcal O}
\newcommand{\cC}{\mathcal C}
\newcommand{\ccC}{\overline{\cC}}

\newcommand{\N}{\mathbb N}

\newcommand{\R}{\mathbb{R}}

\newcommand{\Z}{\mathbb{Z}}

\renewcommand{\H}{\mathbb{H}}

\newcommand{\GL}{\mathrm{GL}}

\newcommand{\SO}{\mathrm{SO}}

\renewcommand{\mod}{{\rm mod\; }}

\newcommand{\be}{\begin{equation}}
\newcommand{\ee}{\end{equation}}

\usepackage[all]{xy}

\usepackage{color}

\usepackage{bm}

\usepackage{bbm}
\renewcommand{\mathbb}{\mathbbm}

\usepackage[margin=1.1in]{geometry}
\usepackage{amsmath,amssymb,latexsym,amsthm}

\renewcommand{\r}{\mathbf{r}}
\renewcommand{\H}{\mathcal{H}} 
\newcommand{\cP}{\mathcal{P}} 
\newcommand{\fe}{\varphi}

\renewcommand{\Z}{\mathbf{Z}}

\renewcommand{\R}{\mathbf{R}}

\renewcommand{\e}{\varepsilon}

\newcommand{\del}{\partial}

\title {Secular dynamics for curved two-body problems} 
\author{Connor Jackman\footnote{Cimat, Guanajuato, Mexico. \today}}\date{}
\begin{document}

\maketitle

\begin{abstract}Consider the dynamics of two point masses on a surface of constant curvature subject to an attractive force analogue of Newton's inverse square law. When the distance between the bodies is sufficiently small, the reduced equations of motion may be seen as a perturbation of an integrable system. We define suitable action-angle coordinates to average these perturbing terms and observe dynamical effects of the curvature on the motion of the two-bodies.
\end{abstract}

\section{Introduction}

When the geometry of constant curvature spaces was brought to attention in the 19th century, mechanical problems in such spaces presented an opportunity for an interesting variation on familiar themes. In particular, studying the motion of point masses subject to extensions of Newton's usual (flat) inverse square force law to analogous `curved inverse square laws'  was soon undertaken, see e.g. the historical reviews in \cite{BMK2, DCS}.  Consequently, we speak of \textit{the} curved Kepler problem or \textit{the} curved $n$-body problem to refer to the mechanics of point masses in a space of constant curvature under such  nowadays customary extensions of Newton's inverse square force. 

While the curved Kepler problem is quite similar to the flat Kepler problem --both being characterized by their conformance to analogues of Kepler's laws, or super-integrability-- there are striking differences between the curved 2-body problem and the usual 2-body problem. Namely, it has long been noticed that the absence of Galilean boosts prohibits straightforward reduction of a curved 2-body problem to a curved Kepler problem, and more recently it has been shown --as an application of Morales-Ramis theory in \cite{Shch}, and by numerically exhibiting complicated dynamics in \cite{BMK}-- that the curved 2-body problem, in contrast to the usual 2-body problem, is non-integrable. In this article we will apply classical averaging techniques to compare some features of certain orbits in these curved and flat 2-body dynamics.

When the two bodies are close, their motion is a small perturbation of the flat 2-body problem. Indeed, after reduction, we find their relative motion governed by a Kepler problem and additional perturbing terms which vanish as the distance between the bodies goes to zero. Consequently their motion at each instant is approximately along an \textit{osculating conic}: the conic section the two bodies would follow if these perturbing terms were ignored. The averaged, or \textit{secular}, dynamics take into account how the perturbing terms slowly deform these osculating conics.

Our main results, theorems \ref{thm:qper} and \ref{thm:per} below, describe quasi-periodic and periodic motions in the reduced curved 2-body problems (see fig.~\ref{fig:orbs}).

\begin{figure}[!htb]
\centering
\includegraphics[scale = .28]{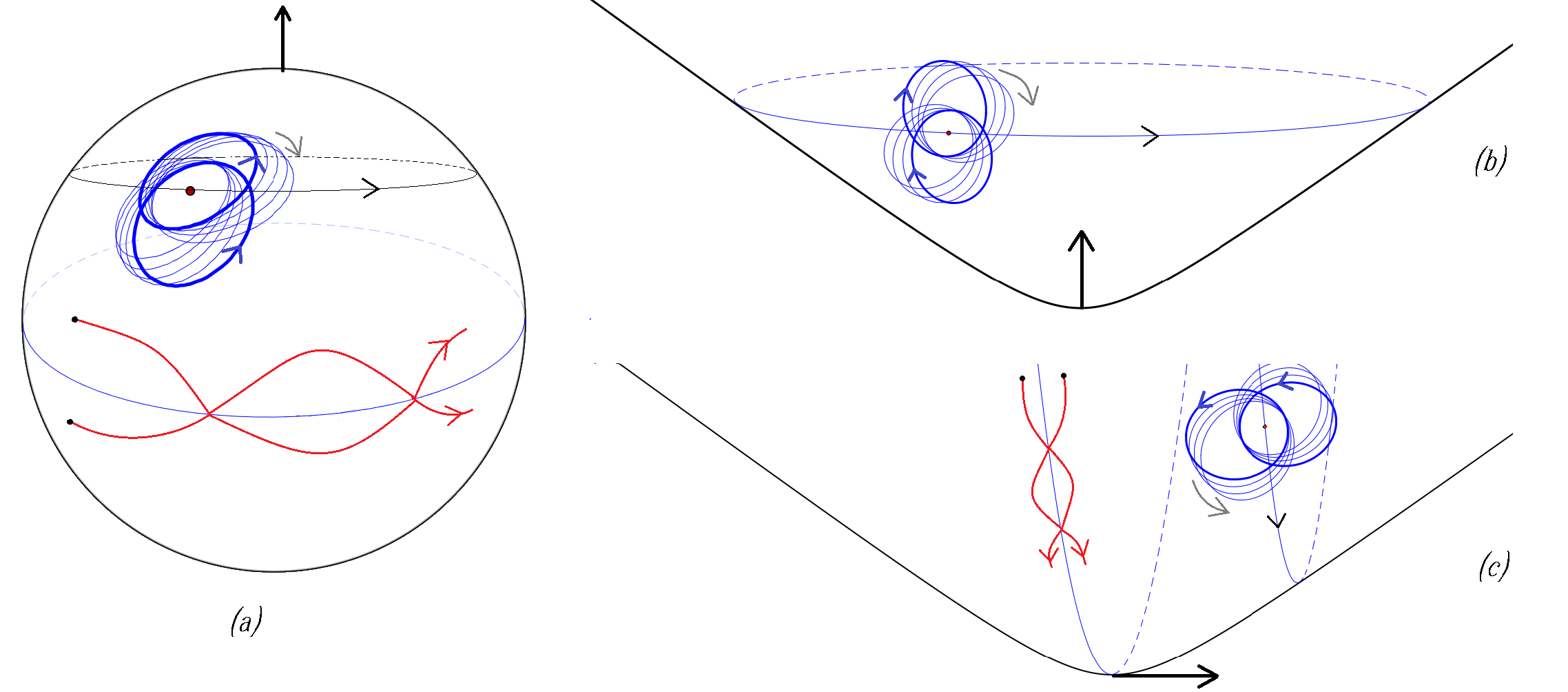}
\caption{\footnotesize{(a) For two sufficiently close bodies on a sphere, there are motions (blue) along which the two bodies revolve nearly along conic sections which precess in the direction \textit{opposite} to the bodies (fast) motion around the conic. The angular momentum here is vertical, and the whole system rotates about the angular momentum axis. As the two bodies get closer to the equator, their rate of precession decreases and the eccentricity of their osculating conic increases until at the equator two types of (regularized) periodic collision orbits survive: one for which the two bodies bounce perpendicularly to the equator (red) and another for which they bounce along the equator. Similar results hold for two bodies in a negatively curved space, but now the precession is in the \textit{same} direction as the bodies motion around the conic (b) with vertical angular momentum and (c) for angular momentum 'to the right'. We expect some chaotic dynamics to occur near the collision orbits (remark \ref{rmrk:split}).}}\label{fig:orbs}
\end{figure}

We will begin in section \ref{sec:Kep} by recalling properties of the curved Kepler problem, with details in appendix \ref{app:DelPoi}.  As an aside, in seeking parametrizations of the curved Kepler orbits useful for averaging, we derive a 'curved Kepler equation' eq. (\ref{eq:KepEq}) that, as far as we can tell, is new. 

The curved 2-body problem admits symmetries by the isometry group of the space of constant curvature. In section \ref{sec:2bdy}, we carry out a Meyer-Marsden-Weinstein (symplectic) reduction to formulate the equations of motion for the reduced curved 2-body problem (prop.~\ref{prop:red}), and define a scaling of the coordinates (prop.~\ref{prop:scal})  allowing us to focus our attention on the dynamics when the two bodies are close relative to the curvature of the space.

In section \ref{sec:SecDyn} we prove our two main results, establishing that the behaviour depicted in fig.~\ref{fig:orbs} holds for 'most' orbits when the two bodies in a given curved space are sufficiently close.

\begin{rmrk}
Our results are related to a remark (in section 4.1 of \cite{BMK}), in that determining the secular dynamics is a preliminary step towards a proof of non-integrability by 'Poincar\'e's methods'.   Upon defining curved analogues of the Delaunay and Poincar\'e coordinates, our methods are quite similar to those applied to the usual 3-body problem in \cite{FejQP}, although the setting here is simpler in the sense that we only need to average over one fast angle.
\end{rmrk}



\section{Curved Kepler problems}
\label{sec:Kep}
Newton's '$1/r$' potential may be characterized by either of the following properties:
\begin{itemize}
\item it is a (constant multiple) of the fundamental solution of  the 3-dimensional laplacian, $\Delta_{\R^3}$

\item the corresponding central force problem has orbits following Kepler's laws.
\end{itemize}
Generalizing these properties to spaces of constant positive (resp.~negative) curvature yields '$\cot\fe$' (resp.~'$\coth\fe$') potentials. We refer to \cite{Al} for a discussion of these long established potentials and their properties using central projection (fig. \ref{fig:CentProj}), presenting some details here in appendix \ref{app:DelPoi}. In particular, the curved Kepler problems are super integrable and on the open subset of initial conditions leading to bounded non-circular motions we have action-angle coordinates, $(L, \ell, G, g)$ where the energy depends only on $L$ (see prop. \ref{prop:DP}).

\begin{figure}[!htb]
\centering
\includegraphics[scale = .5]{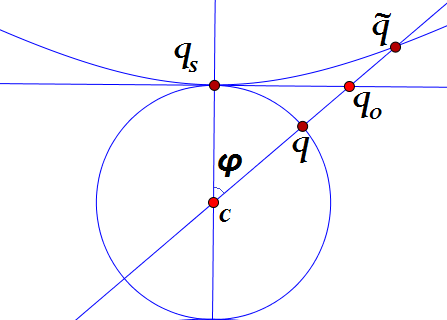}
\caption{\footnotesize{The curved Kepler dynamics --with \textit{fixed} center of attraction or 'sun' at $q_s$-- may be characterized by requiring that its \textit{unparametrized} trajectories centrally project to those of the flat Kepler problem in the tangent plane at $q_s$. Setting $\rho = |cq_s|$, the sphere has curvature $\kappa = 1/\rho^2$ and the hyperboloid (with the restricted Minkowski metric) has curvature $\kappa = - 1/\rho^2$.  We call $\fe := \angle (q_scq)$ (resp.~$\tilde\fe:= \angle(q_sc\tilde q)$ measured with the Minkowski inner product) the \textit{angular distance} from $q$ to $q_s$ (resp.~$\tilde q$ to $q_s$). The distance from $q$ on the sphere to $q_s$ is given by $\rho\fe$, while the distance from $\tilde q$ to $q_s$ on the hyperboloid is given by $\rho \tilde\fe$. We set $r:=|q_oq_s| = \rho \tan\fe = \rho\tanh\tilde\fe$.}}\label{fig:CentProj}
\end{figure}

In what follows, we will consider motions satisfying $\fe = O(\e)$ for all time, where $\e$ is a small parameter. Evidently, for these Kepler problems, such orbits are central projections of elliptic orbits in the plane. Also, to avoid repetitive arguments and computations, we will give details for the positive curvature case and remark on the analogous results for negative curvature.

\section{Curved two body problems}\label{sec:2bdy}

Let $S^2$ be the sphere of curvature $\kappa$, with  $\|\cdot\|_\kappa$ the norm given by its constant curvature metric. Consider two point masses, $$(q_1,q_2) =:q\in (S^2\times S^2)\backslash \{|\cot \fe| = \infty\} =: Q,$$ on this sphere of masses $m_1, m_2>0$ with $\fe$ the angular distance between $q_1$ and $q_2$. Normalize the masses so that $m_1 + m_2 = 1$. We consider the Hamiltonian flow on $T^*Q$ of: $$F := \frac{\|p_1\|_\kappa^2}{2m_1} + \frac{\|p_2\|_\kappa^2}{2m_2} - \frac{m_1m_2}{\rho}\cot\fe.$$

The tangential components of the forces are of equal magnitude and directed towards eachother along the arc $\stackrel{\frown}{q_1q_2}$. By letting $\vec q_j$ be the position vector of $q_j$ from the center of the sphere, the components of the \textit{angular momentum vector}: $$\vec C := m_1\vec q_1\times \dot{\vec q}_1 + m_2 \vec q_1\times \dot{\vec q}_2$$ are first integrals. They correspond to the symmetry by the diagonal action, $(q_1,q_2)\mapsto (gq_1, gq_2)$, of $\SO_3$ on $Q$. Since $\fe\ne 0, \pi$ in $Q$, this $\SO_3$ action on each level set $\{\fe = cst.\}\subset Q$ is free and transitive. By choosing a representative configuration (see fig.~\ref{fig:com}) in each level set of $\fe$ we obtain a slice of this group action and have: \[
    Q\cong I\times \SO_3,~~I = (0,\pi)\ni \fe,\]  such that the $\SO_3$ action is by left multiplication on the second factor.

\begin{figure}[!htb]
\centering
\includegraphics[scale = .5]{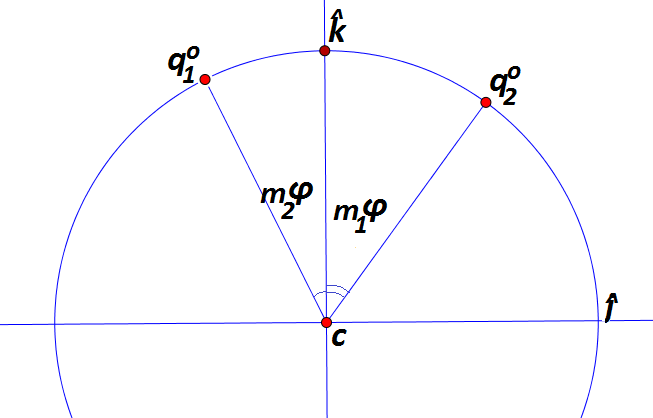}
\caption{\footnotesize{Our choice of representative configuration. Embed the sphere of curvature $\kappa$ in $\R^3$, centered at the origin and take an orthonormal basis $\hat i, \hat j, \hat k$ with $\vec C = C\hat k$. We take $q^o(\fe):=( -\rho\sin m_2\fe \hat j + \rho\cos m_2\fe \hat k, \rho\sin m_1\fe \hat j + \rho\cos m_1\fe \hat k)$. Given $q\in Q$ with angular distance $\fe$ between the two bodies, there is a unique element $g\in \SO_3$ with $q = gq^o(\fe)$.}}\label{fig:com}
\end{figure}


\subsection{Reduced equations of motion} An application of symplectic reduction (see e.g. \cite{ACM} App. 5) gives:

\begin{prop}\label{prop:red} Fix the angular momentum $\vec C\ne 0$. The reduced 2-body dynamics on a sphere of curvature $\kappa = \frac{1}{\rho^2}$ takes place in $T^*I\times\mathcal{O}_\mu$, where $\mathcal{O}_\mu$ is a co-adjoint orbit in $\mathfrak{so}_3^*$. It is the Hamiltonian flow of:
$$F_{red} = Kep_\kappa + \kappa(\frac{\|\vec C\|^2}{2} - p_\theta^2) + O(\fe),~~~dp_\fe\wedge d\fe + dp_\theta\wedge d\theta,$$ for $|p_\theta|<\|\vec C\|$, and with $Kep_\kappa = \frac{p_\fe^2 + p_\theta^2/\sin^2\fe}{2\rho^2m_1m_2} - \frac{m_1m_2}{\rho}\cot\fe$, the Hamiltonian for a curved Kepler problem.
\end{prop}

\begin{proof} We first find the mass weighted metric, $\langle (\vec v_1, \vec v_2), (\vec u_1, \vec u_2)\rangle = m_1 \vec v_1\cdot \vec u_1 + m_2 \vec v_2\cdot \vec u_2$, in terms of our identification $Q\cong I\times \SO_3$ (fig.~\ref{fig:com}). Since the metric is invariant under left translations, it suffices to determine it at the identity (our representative configurations, $q^o$, in fig.~\ref{fig:com}).

Let $X_1, X_2, X_3$ be the infinitesimal symmetry vector fields generated by rotations about the $\hat i, \hat j, \hat k$ axes. Together with $\del_\fe$, they frame $Q$. We compute:
\begin{align*}
\langle \del_\fe, \del_\fe\rangle_{q^o} &= \rho^2m_1m_2 & \langle X_1, X_1\rangle_{q^o} &= m_1\|\hat i \times \vec q_1^o\|^2 + m_2\|\hat i \times \vec q_2^o\|^2 = \rho^2\\
\langle X_2, X_2\rangle_{q^o} &= \rho^2(m_1\cos^2m_2\fe + m_2\cos^2m_2\fe) & \langle X_3, X_3\rangle_{q^o} &= \rho^2(m_1\sin^2m_2\fe + m_2\sin^2m_2\fe)
\end{align*}
 \[   \langle X_2, X_3\rangle_{q^o} = \rho^2(m_1\cos m_2\fe\sin m_2\fe - m_2\cos m_1\fe\sin m_1\fe), \]
 all other inner products being zero. In terms of $Q$'s co-frame, $\{d\fe, X^j\}$, dual to the above frame, the metric is given by inverting its $\{\del_\fe, X_j\}$ matrix representation:
 \begin{align*}
\langle d\fe, d\fe\rangle_{q^o} &= \frac{1}{\rho^2m_1m_2} & \langle X^1, X^1\rangle_{q^o} &= \frac{1}{\rho^2}\\
\langle X^2, X^2\rangle_{q^o} &= \frac{m_1\sin^2m_2\fe + m_2\sin^2m_2\fe}{m_1m_2\rho^2\sin^2\fe} & \langle X^3, X^3\rangle_{q^o} &= \frac{m_1\cos^2m_2\fe + m_2\cos^2m_2\fe}{m_1m_2\rho^2\sin^2\fe}
\end{align*}
 \[   \langle X^2, X^3\rangle_{q^o} = \frac{m_2\cos m_1\fe\sin m_1\fe - m_1\cos m_2\fe\sin m_2\fe }{m_1m_2\rho^2\sin^2\fe}. \]
 
We identify $T^*Q\cong T^*I\times (\SO_3\times \mathfrak{so}_3^*)$ by \textit{right} translation: $\alpha_g\in T_g^*SO_3\mapsto (g, R_g^*\alpha_g)$. In these coordinates, the symplectic lift of $\SO_3$'s left action on $Q$ is given by: $g\cdot (\fe, p_\fe, h, \mu) = (\fe, p_\fe, gh, Ad_{g^{-1}}^*\mu)$ with associated moment map: $J(\fe, p_\fe, g, \mu) =  \mu$.  The reduced dynamics takes place on the quotient $P_\mu := J^{-1}(\mu)/G_\mu\cong T^*I\times\mathcal{O}_\mu$, realized by:  $(\fe, p_\fe, g, \mu)\mapsto (\fe, p_\fe, Ad_g^*\mu)$. The reduced symplectic form on $P_\mu$ is $dp_\fe\wedge d\fe + \Omega$, where $\Omega_\nu(ad_\xi^*\nu, ad_\eta^*\nu) = \nu([\xi, \eta])$ is the Kirillov-Kostant form on $\mathcal{O}_\mu$.

By left invariance of the two body Hamiltonian, $F$, we have:  $F(\fe, p_\fe, g, \mu) = F(\fe, p_\fe, e, Ad_g^*\mu)$, or letting $\nu = Ad_g^*\mu = \nu_j X^j(q^o)$, and $m = m_1m_2$, we have:  \begin{equation}\label{eq:Fred} F_{red} = \frac{p_\fe^2}{2m\rho^2} - \frac{m}{\rho}\cot\fe + \frac{\nu_1^2}{2\rho^2} + \nu_2^2\frac{m_1\sin^2m_2\fe + m_2\sin^2m_1\fe}{2m\rho^2\sin^2\fe}\end{equation} \[ + \nu_3^2\frac{m_1\cos^2m_2\fe + m_2\cos^2m_1\fe}{2m\rho^2\sin^2\fe} + \nu_2\nu_3\frac{m_2\cos m_1\fe\sin m_1\fe - m_1\cos m_2\fe\sin m_2\fe}{m\rho^2\sin^2\fe}.\]

Now, using that $\frac{\sin^2ax}{\sin^2x}$ is an analytic function around $x=0$, with expansion $a^2 + O(x^2)$, and that the $\nu_2\nu_3$ coefficient is analytic around $\fe = 0$ of $O(\fe)$, and $\nu_1^2 + \nu_2^2 + \nu_3^2 = \|\vec C\|^2$, we have:

\[ F_{red} = \frac{p_\fe^2}{2m\rho^2} + \frac{\nu_3^2}{2m\rho^2\sin^2\fe} - \frac{m}{\rho}\cot\fe + \kappa(\frac{\|\vec C\|^2}{2} - \nu_3^2) + O(\fe).\]

Finally, recall that the Kirillov-Kostant form is given by $\Omega = dp_\theta \wedge d\theta$ when we parametrize $\mathcal{O}_\mu$ -- away from the poles $|p_\theta| = \|\vec C\|$ -- as $(\sqrt{\|\vec C\|^2 - p_\theta^2}\sin\theta, \sqrt{\|\vec C\|^2 - p_\theta^2}\cos\theta, p_\theta) = (\nu_1, \nu_2, \nu_3)$.
\end{proof}

\begin{rmrk}\label{rmrk:eqMass} For equal masses, the expression for $F_{red}$ may be simplified by taking $\psi = \fe/2, p_\psi = 2p_\fe$. One has: $F_{red} = Kep_\kappa + \frac{\tan\psi}{8\rho} + \kappa\frac{\|\vec C\|^2 - p_\theta^2}{2}(1 + (\tan\psi\cos\theta)^2)$, where $Kep_\kappa = \frac{1}{2\rho^2}(p_\psi^2 + \frac{p_\theta^2}{\sin^2\psi}) - \frac{\cot\psi}{8\rho}$. 
\end{rmrk}

\begin{rmrk} In \cite{Shch}, symplectic reduction was also applied to reduce the curved 2-body problems. However, for purposes of comparison, it was not clear to us what representative configuration or slice of the group action was used, which led us to present a reduction here 'from scratch'.  Our choice of representative (fig.~\ref{fig:com}) was motivated by the mass metric being more diagonal, namely requiring: $\langle \del_\fe, X_j\rangle = 0$.
\end{rmrk}

Next, we introduce a small parameter, $\e$, to study the dynamics when $\fe = O(\e)$:

\begin{prop}\label{prop:scal}
The dynamics on the open set $\fe = O(\e)$ may be reparametrized as that of:
$$\hat F_{red} = Kep_{\e^2} +  Per,~~\hat\omega = d\hat L\wedge d\hat\ell + d\hat G\wedge d\hat g$$ where $Kep_{\e^2} = - \frac{m^3}{2\hat L^2} + \frac{\e^2 \hat L^2}{2m}$, $m = m_1m_2$, and $Per = \e^2\left(\frac{\hat C^2}{2} - \hat G^2 + (m_2-m_1)O(\e) + O(\e^2)\right)$.
\end{prop}

\begin{proof} Take Delaunay coordinates (app. \ref{prop:DP}), $(L,G,\ell, g)$,  for the $Kep_\kappa$ term of $F_{red}$.  Let $L^2 =: \rho\e\hat L^2, G^2 =: \rho\e\hat G^2, \|\vec C\|^2 =: \rho\e \hat C^2$, so that $\fe = O(\e)$. The dynamics of $\hat F_{red} := \rho\e F_{red}$ with symplectic form $\hat\omega := d\hat L\wedge d\ell + d\hat G\wedge dg = \omega/\sqrt{\rho\e}$, corresponds to the time reparametrization: $(\rho\e)^{3/2}\hat t = t$, where $\hat t$ is the new time. Note that the mean anomaly, $\hat\ell$, for $Kep_{\e^2}$ is the same as the mean anomaly, $\ell$, for the $Kep_{\kappa}$ term, as can be easily checked by comparing $d\ell = \del_LKep_\kappa~dt$ and $d\hat\ell = \del_{\hat L} Kep_{\e^2}~d\hat t$.
\end{proof}

\begin{rmrk} The scaling above includes the possibility of taking $\e = 1/\rho$, and in place of imagining the bodies as close, view their distance remaining bounded as the curvature of the space goes to zero.
\end{rmrk}

\begin{rmrk}
When the curvature is negative, eq. (\ref{eq:Fred}) has the trigonometric functions replaced by their hyperbolic counterparts, and with   $\mathfrak{so}_{2,1}^*$'s coadjoint orbits:  $\nu_1^2 + \nu_2^2 - \nu_3^2 = \|\vec C\|_{2,1}^2$ replacing the coadjoint orbits of $\mathfrak{so}_3^*$.  One arrives at $\hat F_{red} = Kep_{-\e^2} + \e^2(\frac{\|\hat C\|_{2,1}^2}{2} + \hat G^2 + (m_2-m_1)O(\e) + O(\e^2))$.
\end{rmrk}

\section{Secular dynamics}
\label{sec:SecDyn}

The \textit{secular} or averaged dynamics refers to the dynamics of the Hamiltonian: \[\langle F_{red}\rangle := \frac{1}{2\pi}\int_0^{2\pi} F_{red}(L,\ell, G, g)~d\ell. \] The secular dynamics is integrable, $L$ being a first integral. We will use the brackets to denote a functions averaged over $\ell$, namely $\langle f\rangle = \frac{1}{2\pi}\int_0^{2\pi} f~d\ell$. The dynamical relevance of the secular Hamiltonian is due to its appearance in a series of 'integrable normal forms' for the reduced curved 2-body dynamics:

\begin{prop} Consider the scaled Hamiltonian, $(\hat F_{red}, \hat\omega)$, of proposition \ref{prop:scal}. For each $k\in\N$, there is a symplectic change of coordinates, $(\psi^k)^*\hat\omega = \hat\omega$, with:
\[\hat F_{red}\circ\psi^k = Kep_{\e^2} +  \langle Per\rangle + \langle F^1\rangle + ... + \langle F^{k-1}\rangle + F^{k}\]
and each $F^j = O(\e^{2j+3})$. When the masses are equal, $F^j = O(\e^{2j + 4})$.
\end{prop}
\begin{proof} The iterative process to determine the sequence of $\psi^k$ and $F^k$'s is not new (see \cite{FejQP}). For the first step, let $\psi^1$ be the time 1-flow of a 'to be determined' Hamiltonian $\chi$. By Taylor's theorem:
\[ \hat F_{red}\circ \psi^1 = Kep_{\e^2} + Per + \{Kep_{\e^2}, \chi\} + \{Per, \chi\} + \int_0^1 (1-t)\frac{d^2}{dt^2}(\hat F_{red}\circ\psi^t)~dt.\]
Let $\tilde Per := Per - \langle Per\rangle$. Note that, while $Per = O(\e^2)$, we have $\tilde Per = O(\e^3)$. Taking $\hat n \chi := \int_0^\ell \tilde Per(\hat L, l, \hat G, g)~dl,$ where $\hat n = \del_{\hat L}Kep_{\e^2}$, we have: \[ Per + \{Kep_{\e^2}, \chi\} = \langle Per\rangle,~~\chi = O(\e^3),\] so that for $\e$ sufficiently small $\psi^1$ is defined. Finally, integrating by parts, one has \[ F^1 := \{Per, \chi\} + \int_0^1 (1-t)\frac{d^2}{dt^2}(\hat F_{red}\circ\psi^t)~dt = \int_0^1 \{\langle Per\rangle + t\tilde Per, \chi\}\circ \psi^t~dt = O(\e^5).\]
\end{proof}

\begin{rmrk}\label{rmrk:kSec} The truncated, or \textit{$k$'th order secular system}: $Sec^k:=Kep_{\e^2} +  \langle Per\rangle + \langle F^1\rangle + ... + \langle F^{k-1}\rangle$ is integrable, since it does not depend on $\ell$. If the sequence $\psi^k$ converged, the curved 2-body problem would be integrable, although by \cite{Shch} we know there is no such convergence. For each fixed $k$ however, the $k$'th order secular system approximates the true motion over long time scales in a region of configurations with the 2-bodies sufficiently close.
\end{rmrk}

In appendix \ref{app:Per}, we compute an expansion of $\langle Per\rangle$ in powers of $\e$, eq. (\ref{eq:PerExp}). Via this expansion, one finds two unstable periodic orbits (red in fig.~\ref{fig:orbs}), whose stable and unstable manifolds coincide for the integrable secular dynamics. The complicated dynamics which would result from the splitting of these manifolds is a likely mechanism responsible for the non-integrability of the curved 2-body problem.

By ignoring the $O(\fe)$ terms in $F_{red}$, the Keplerian orbits experience precession (see fig.~\ref{fig:orbs}): \[ \dot G = 0,~~\dot g = - 2\kappa G.\] When the curvature is positive, the conic precesses in a direction opposite to the particles (fast) motion around the conic, while when the curvature is negative the precession occurs in the same direction as the particles motion. We will check that for $\fe$ sufficiently small, many such orbits survive the perturbation of including the $O(\fe)$ terms: they may be continued to orbits of the reduced curved 2-body problem and then lifted to orbits of the curved 2-body problem.

\subsection{Continued orbits}

Applying a KAM theorem from \cite{Fej}, establishes certain quasi-periodic motions of the reduced problem. Introducing some notation, for $\gamma>0, \tau>1$, let \[ D_{\gamma,\tau}:= \{ v\in\R^2~:~|\mathbf{k}\cdot v|\ge \frac{\gamma}{(|k_1|+|k_2|)^\tau}, \forall \mathbf{k} = (k_1,k_2)\in\Z^2\backslash \mathbf{0}\} \] be the $(\gamma,\tau)$-\textit{Diophantine frequencies}. Take \[\mathbf{\hat B}:= \{ (\hat L, \hat G) ~:~ 0 < \hat G \le \hat L < 1\} \] and let \[(I,\theta) = (I_1,I_2,\theta_1, \theta_2) = (\hat L, \hat G, \ell, g) + O(\e^3)\] be action-angle variables for some $k$th order secular system, with $\mathbf{B} = \{ (I_1, I_2) :  (\hat L, \hat G)\in \mathbf{\hat B}\}\subset \R^2$ equipped with the Lebesgue measure, $Leb$, restricted to $\mathbf{B}$. 

The reduced curved 2-body problem has, for $\fe$ sufficiently small,  a positive measure of invariant tori along each of which the motions are quasi-periodic. More precisely:

\begin{theorem}\label{thm:qper}
Fix $\tau >1, \hat\gamma >0, m>2$ and set $\gamma := \e^m\hat\gamma$. There exists $\e_0>0$ such that for each $0<\e<\e_0$, we have:

(i) A positive measure set, $d_{\gamma, \tau}\subset \mathbf{B}$,

(ii) for each $I^o\in d_{\gamma,\tau}$, there is a local symplectic change of coordinates, $\psi^o$, for which $\psi^o(I^o,  \theta)$ is an invariant torus of $\hat F_{red}$ with frequency in $D_{\gamma,\tau}$.

(iii) As $\e\to 0$, $Leb(d_{\gamma,\tau})\to Leb(\mathbf{B})$ and $\psi^o\to id$ (in the Whitney $C^\infty$ topology).
\end{theorem}
\begin{proof} We verify the hypotheses of \cite{Fej}.  Let $(I,\theta)$ be action-angle variables for a $k$'th order secular system (remark \ref{rmrk:kSec}) with $2k>2m-3$. Setting $s := I^o\in \mathbf{B}$, and $r := I - I^o$, we have by Taylor expansion:
\[ F_s^o := Sec^k = c_s^o + \alpha_s^o\cdot r + O(r^2) \] where $c_s^o := Sec_k(I^o), \alpha_s^o := dSec^k(I^o) = O(1, \e^2)$. Likewise, we take $F_s:=\hat F_{red}(r,\theta)$ around $s = I^o$. Since we have chosen $2k + 1>2m$, we have \[|F_s - F_s^o| = o(\gamma^2)\] for $\e$ sufficiently small. So  (\cite{Fej} Thm.~15 and remark 21) there exists a map $\mathbf{B}\ni s\mapsto \alpha_s\in \R^2$, with $|\alpha_s - \alpha_s^o|<<1$ in the $C^\infty$-Whitney topology such that \textit{provided} $\alpha_s\in D_{\gamma,\tau}$, we have $\hat F_{red}\circ\psi^o = c_s + \alpha_s\cdot r + O(r^2;\theta)$ for some local symplectic map $\psi^o$. It remains to check that $d_{\gamma,\tau}:=\{ s~:~\alpha_s\in D_{\gamma,\tau}\}$ is non-empty. Since $m>2$ and $\alpha_s = O(1, \e^2)$ we have (\cite{Fej} Cor.~29) \[ Leb(\mathbf{B}\backslash d_{\gamma,\tau}) = O(\e^{\frac{m-2}{\mu}})\] for some constant $\mu>0$. So indeed, for $\e$ sufficiently small, $d_{\gamma,\tau}$ has positive measure and is non-empty.
\end{proof}

\begin{rmrk} As the reduced dynamics has two degrees of freedom, there is KAM-stability in the $\fe<<1$ regime: the invariant tori form barriers in energy level sets.
\end{rmrk}

One may also apply the implicit function theorem to establish periodic orbits of long periods for the reduced dynamics:

\begin{theorem}\label{thm:per}
Let $m,n\in\N$. Then for $m$ sufficiently large, there exist periodic orbits of the reduced curved 2-body problem for which the bodies make $m$ revolutions about their osculating conics, while the osculating conic precesses around the origin $n$ times.
\end{theorem}
\begin{proof}
We will consider non-equal masses, a similar argument applies for equal masses. After iso-energetic reduction, the equations of motion for $\hat G, g$ are eq. (\ref{eq:PerExp}): 
\[ \frac{d\hat G}{d\ell} = \e^3\mathfrak{m} \sin g + O(\e^4),~~~\frac{dg}{d\ell} = -2\e^2\hat G + O(\e^3),\] where $\mathfrak{m}(\hat L, \hat G)\ne 0$ for non-circular motions. Consider the 'long time return map':
\[ \overline P_\e (G^o, g^o) :=  \int_0^{2\pi/\e^2} (\frac{1}{\e}\frac{d\hat G}{d\ell}, ~\frac{dg}{d\ell})~d\ell,\] the integral taken over an orbit with initial condition $G^o, g^o$.  When $\overline P_\e(G^o, g^o) = (0, -2\pi n)$ and $1/\e^2 = m\in\N$, we have a periodic orbit as described in the theorem. The map $\overline P_\e$ is analytic in $\e$, with
 \[\overline P_0(G^o, g^o) = (\mathfrak{m}\sin g^o, ~ - 2G^o).\]
By the implicit function theorem, $g^o\equiv 0~\mod \pi, G^o = \pi n$, continues to solutions of $\overline P_\e(G^\e, g^\e) = (0, -2\pi n)$ for $0<\e<\e_o$, yielding periodic orbits of long periods for those $\e$ with  $ 1/\e^2 = m\in\N$.
\end{proof}

\begin{rmrk}
When the curvature is positive the precession is, as with the quasi-periodic motions, in the opposite direction to the particles motion around the conic, while for negative curvature in the same direction.
\end{rmrk}

\subsection{Lifted orbits}

The continued orbits are the image under a near identity symplectic transformation of certain orbits of $F_{red}$ with the $O(\fe)$ terms neglected: \[ H_{red} :=  Kep_\kappa + \kappa(\frac{\|\vec C\|^2}{2} - G^2). \] We describe how the precessing Keplerian orbits of $H_{red}$ lift to $T^*Q$.

In the reduced space, $T^*I\times\mathcal{O}_\mu$, the orbits of $H_{red}$ are given by:
\[ (p_\fe(t), \fe(t), Ad_{g(t)}^*\mu)\]
where $g(t) = R_{\hat i}(\lambda)R_{\hat k}(\theta(t))$ is a rotation by $\theta(t)$ about the $\hat k$-axis followed by a rotation by $\lambda$ about the $\hat i$-axis. Moreover, $(\fe(t), \theta(t))$ describe a precessing orbit of the curved Kepler problem having $G = \|\vec C\|\cos \lambda$ for its angular momentum.

We take $\vec C = C\hat k$ (fig.~\ref{fig:com}), so that the isotropy is by rotations about the $\hat k$-axis. The ambiguity in the reduced orbit is $Ad_{g(t)}^*Ad_{h(t)}^*\mu$ where $h(t) = R_{\hat k}(\omega(t))$. Hence the lifted orbit, on $T^*Q$, is of the form:
\[(*)~~~ (\fe(t), p_\fe(t),  h(t)g(t), \mu).\] 

The Hamiltonian $H_{red}$ is the symplectic reduction of $H$ on $T^*Q$ given by $H(\fe, p_\fe, g, \mu) = H_{red}(\fe, p_\fe, Ad_g^*\mu)$. Requiring the lifted curve $(*) $ to satisfy the equations of motion of $H$, imposes the condition: \[\dot\omega = \kappa \|\vec C\|.\]

In summary, a precessing Keplerian orbit with angular momentum $G$ lifts to an orbit with angular momentum $\vec C = C\hat k$ on $Q$ where the two bodies follow precessing Keplerian orbits centered (we use the notion of 'center' from fig.~\ref{fig:com} here) on the colatitude, $\lambda$, satisfying $C\cos\lambda = G$. The whole system is rotated with angular speed $\kappa C$ about the $\hat k$ axis. 

\begin{rmrk}
As $\lambda\to \pi/2$, the center of the system approaches the equator and the bodies move along more eccentric conic sections until at the equator we have collision orbits. The rate of precession decreases as one moves closer to the equator, where the collision orbits cease to precess (see remark \ref{rmrk:split}, for a finer description of the equatorial behaviour). When $C > L$, there is a maximal value of $\lambda$, and such orbits are constrained to bands around the equator.
\end{rmrk}

The true motions of the curved 2-body problem established in the previous section are qualitatively the same as such lifted orbits of $H_{red}$, performing small oscillations around them.



\appendix

\section{Action-angle coordinates for the curved Kepler problem}
\label{app:DelPoi}

\subsection{Curved conics}

The Kepler problem on a sphere of radius $\rho$ with a fixed 'sun', $q_s$, of mass $M$ and particle, $q\in S^2$, of mass $m$ is given by the Hamiltonian flow of:
\[    Kep_\kappa := \frac{\|p\|_\kappa^2}{2m} - \frac{mM}{\rho}\cot\fe\] on $T^*(S^2\backslash \{|\cot\fe| = \infty\})$. Here $\|\cdot \|_\kappa$ is the norm induced by the metric of constant curvature $\kappa = 1/\rho^2$ on $S^2$ and $\fe$ is the angular distance from $q_s$ to $q$ (see figure \ref{fig:CentProj}). For a negatively curved space, one replaces $\cot\fe$ with $\coth\fe$.

Because the force is central, letting $\vec q\in\R^3$ be the position of the particle from the center, $c$, of the sphere and $\hat k$ a unit vector from $c$ to $q_s$, the \textit{angular momentum}: \begin{equation}\label{eq:am}
    G:=m(~\vec q\times \dot{\vec q}~)\cdot \hat k
\end{equation} is a first integral. It corresponds to the rotational symmetry about the $\Vec{cq_s}$ axis.

\begin{rmrk}\label{rem:coords}
In spherical coordinates, $\vec q = \rho (\sin\fe\cos\theta, \sin\fe\sin\theta, \cos\fe)$, we have:
$\|p\|_\kappa^2 = \frac{1}{\rho^2}(p_\fe^2 + \frac{p_\theta^2}{\sin^2\fe})$, and $G = m\rho^2\sin^2\fe\dot\theta = p_\theta$.
In these coordinates, it is not hard to show analytically that the curved Kepler trajectories centrally project to flat Kepler orbits (fig.~\ref{fig:CentProj}). Indeed,  setting $R := 1/r = \frac{\cot\fe}{\rho}$, one finds $\frac{d^2R}{d\theta^2} + R = m^2M/G^2$, so that the $\theta$-parametrized orbits are: $r = \frac{G^2/m^2M}{1 + e\cos(\theta - g)}$, with $e$ and $g$ being constants of integration. Here $\nu := \theta - g$ is the true anomaly and $g$ is the \textit{argument of pericenter}.
\end{rmrk}

The curved Kepler trajectories are in fact  conic sections on the sphere having a focus at $q_s$ (see fig.~\ref{fig:curvedConic}). One may express energy and momentum in terms of geometric parameters of such spherical conics.

\begin{figure}[!htb]
\centering
\includegraphics[scale = .4]{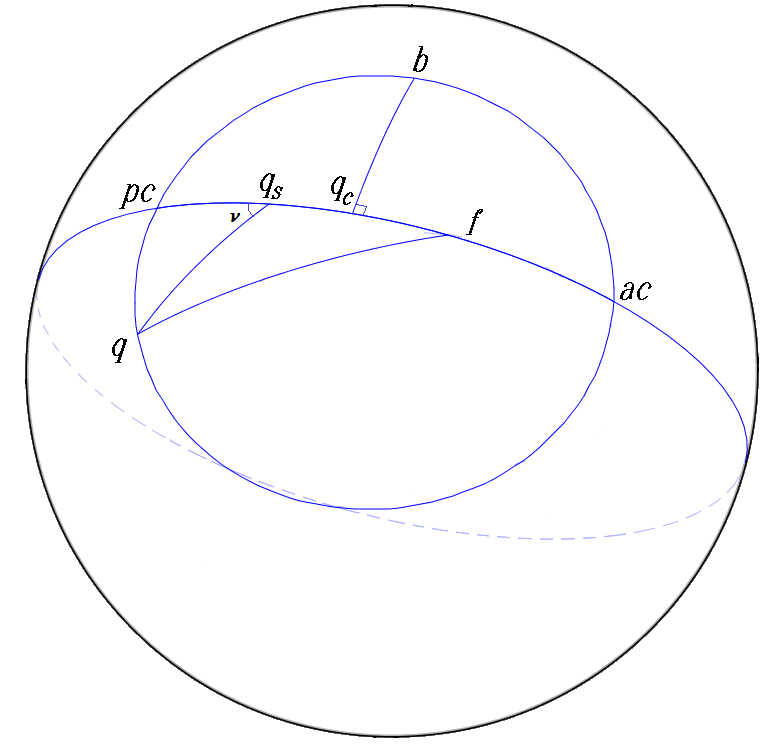}
\caption{\footnotesize{An ellipse on the sphere, with foci at $q_s$ and $f$ is the set of points $q$ for which $|qq_s| + |qf| = 2\alpha$ is constant. We call $\alpha$ its \textit{semi-major axis}, the midpoint, $q_c$, between the foci its \textit{center}, and the length $|bq_c| =:\beta$ its \textit{semi-minor axis}. The number, $\epsilon$, for which $|q_sq_c| = \alpha\epsilon$ is called the \textit{eccentricity}. When $\epsilon\ne 0$, the closest point to $q_s$ is called the \textit{pericenter}, $pc$, and furthest, $ac$, the \textit{apocenter}. The angle $\nu := \angle(pc,q_s,q)$ is called the \textit{true anomaly} of $q$.}}\label{fig:curvedConic}
\end{figure}

\begin{prop}\label{prop:geomcoords}(see \cite{Al})\label{prop:enmom} Consider an orbit of the curved Kepler problem along a spherical conic with a focus at $q_s$ (fig.~\ref{fig:curvedConic}). Let $\alpha$ be the semi-major axis and $\beta$ the semi-minor axis of this conic. The orbit has energy and momentum:
\[Kep_\kappa = -\frac{mM}{\rho}\cot\frac{2\alpha}{\rho},\]
\[G^2 = m^2M\rho\tan^2\frac{\beta}{\rho}\cot\frac{\alpha}{\rho}.\]
\end{prop}

\begin{proof}
Consider the spherical triangle $\Delta(f,q_s,q)$ with sidelengths $2\alpha\epsilon, \rho\fe, 2\alpha - \rho\fe$ and interior angle $\pi - \nu$ opposite to side $\stackrel{\frown}{fq}$ (see fig.~\ref{fig:curvedConic}). The cosine rule of spherical trigonometry yields: $r = \rho\tan\fe = \frac{p^2}{1 + e\cos\nu},$ where $p^2 = \rho\frac{\cos\frac{2\alpha\epsilon}{\rho} - \cos\frac{2\alpha}{\rho}}{\sin\frac{2\alpha}{\rho}}$ and $e = \frac{\sin\frac{2\alpha\epsilon}{\rho}}{\sin\frac{2\alpha}{\rho}}$, so that the orbits are indeed curved conic sections. Comparing with the expression in remark \ref{rem:coords} yields $G^2 = m^2Mp^2$. The expression for $G$ in the proposition follows by using the relation: $\cos\frac{\beta}{\rho} = \frac{\cos\frac{\alpha}{\rho}}{\cos\frac{\alpha\epsilon}{\rho}}$ (consider the right spherical triangle $\Delta(q_s,q_c,b)$), to simplify.

To obtain the expression for $Kep_\kappa$, observe that $\fe_p := \frac{\alpha(1-\epsilon)}{\rho}$ and $\fe_a := \frac{\alpha(1+\epsilon)}{\rho}$ are maximal and minimal values of $\fe$ over the trajectory having energy $Kep_\kappa =: h$. Consequently $p_\fe = 0$ at $\fe_{a,p}$ so we have two solutions of the equation: $h\cos 2\fe = \frac{mM}{\rho}\sin 2\fe + h - \frac{G^2}{\rho^2}.$ Adding and subtracting the above equation evaluated at $\fe_{a,p}$ yields: $Kep_\kappa = h = \frac{mM}{\rho}\frac{\sin 2\fe_a - \sin 2\fe_p}{\cos 2\fe_a - \cos 2\fe_p} = - \frac{mM}{\rho}\cot(\fe_a + \fe_p).$ 
\end{proof}

\begin{rmrk}
The sign of $G$ represents the orbits orientation, with positive $G$ for counterclockwise motion around $q_s$. The same arguments apply to bounded motions when the curvature is negative, replacing trigonometric functions with their hyperbolic counterparts. The energy for bounded motions in a space of curvature $\kappa = - 1/\rho^2$ is always less than $-mM/\rho$.
\end{rmrk}

\begin{rmrk}
The orbits of the curved Kepler problem colliding with $q_s$ may be 'regularized' similarly to the usual Kepler regularization, via an elastic bounce  \cite{ChavReg}.
\end{rmrk}

\subsection{Delaunay and Poincar\'e coordinates}

We have the following analogues of the Delaunay and Poincar\'e symplectic coordinates for the curved Kepler problem:

\begin{prop}\label{prop:DP} Let $G$ be the angular momentum, $g$ the argument of pericenter,
\[L^2 := m^2M\rho\tan\frac{\alpha}{\rho},~~L>0,\]  and $\ell$ be proportional to the area swept out by the \textit{orthogonal} projection along the conic from pericenter, scaled so that $\ell\in \R/2\pi\Z$. Then $(L, \ell, G, g)$ are symplectic (Delaunay) coordinates for bounded non-circular motions.  The variables: \[\Lambda = L, ~~\lambda = \ell + g, ~~\xi = \sqrt{2(L-|G|)}\cos g, ~~\eta = \sqrt{2(L-|G|)}\sin g,\] are symplectic (Poincar\'e) coordinates in a neighborhood of the circular motions.

The energy is given by:
\begin{equation}\label{eq:KepL} Kep_\kappa = -\frac{m^3M^2}{2L^2} + \kappa\frac{L^2}{2m}.\end{equation}
\end{prop}

\begin{proof}
The construction of these coordinates is almost identical as for the planar Kepler problem (see e.g. \cite{FejAA}). We only found some difference in the computation determining $L$, owing to the fact that in the curved case, the period as a function of energy (Kepler's third law) does not have such a simple expression.

For non-circular motions, we have symplectic coordinates $(H = Kep_\kappa, t, G, g)$, where $t$ is the time. Since orbits of fixed energy, $Kep_\kappa = H$, all have a common period, $T(H)$, we set $\ell = \frac{2\pi}{T(H)}t$ and seek a conjugate coordinate $L(H)$ to $\ell$, i.e. we want to integrate: 
\[dL = \frac{T(H)}{2\pi}dH.\]
To integrate this expression, we make some changes of variable. For $H$ an energy value admitting bounded motions, let $\fe_c(H)$ be the angular distance of the circular orbit having energy $H$. Then: \[ H = - \frac{mM}{\rho}\cot 2\fe_c,~~G_cT(H) = 2\pi m\rho^2\sin^2\fe_c\] where $G_c^2 := \rho m^2M\tan\fe_c$ is the angular momentum of the circular solution. Note that  \[(*)~~~~H = -\frac{m^3M^2}{2G_c^2} + \kappa \frac{G_c^2}{2m}.\] We compute that $\frac{T}{2\pi}dH = dG_c$. So we take $L = G_c$, and have $L^2 = m^2M\rho\tan\frac{\alpha}{\rho}$ from $(*)$ and prop. \ref{prop:enmom}.
\end{proof}

\begin{rmrk}
The mean anomaly, $\ell$, is related to time by $d\ell = n~ dt$ where $n(L) := \frac{m^3M^2}{L^3} + \kappa\frac{L}{m}$. When the curvature is negative, the same arguments lead to $L^2 = m^2M\rho\tanh\frac{\alpha}{\rho}$, and the same expression, eq. (\ref{eq:KepL}), for the energy.
\end{rmrk}

\subsection{Other anomalies}\label{sec:anoms}

In averaging functions over curved Keplerian orbits, e.g.: determining $\frac{1}{2\pi}\int_0^{2\pi} f(q)~d\ell$, it is often useful to perform a change of variables, as the position on the orbit, $q$, does not have closed form expressions in terms of $\ell$. We collect here some parametrizations of Keplerian orbits.  Although not all are necessary for our main results, they may serve useful in other perturbative studies of the curved Kepler problem. 

The position on the conic is given explicitely in terms of the true anomaly (remark \ref{rem:coords}). By conservation of angular momentum (recall $d\ell = n~dt$ with $ n = \frac{m^3M^2}{L^3} + \kappa \frac{L}{m}$): 
\[ G~d\ell =n~m\rho^2\sin^2\fe ~d\nu. \]

One may centrally project a curved Kepler conic to the tangent plane at $q_s$ and then parametrize this planar conic by its eccentric anomaly, which we denote here by $u_o$. Letting $a, e, b$ be the semi-major axis, eccentricity and minor axis of this planar conic, the position is given by:
\begin{align}\label{eq:flatEApos}
r &= \rho\tan\fe = a(1 - e\cos u_o) & x &= r\cos\nu = a(\cos u_o - e) & y&=r\sin\nu = b\sin u_o 
\end{align}
 and one has: \begin{equation}\label{eq:flatEA}
    \sqrt{\frac{M}{a}}~d\ell =n~\rho\sin\fe\cos\fe~du_o.\end{equation}

Some more parametrizations arise naturally when one \textit{orthogonally} projects the curved Kepler ellipse onto the tangent plane at $q_s$. The time-parametrized motion along this  plane curve sweeps out area at a constant rate, however it is now a quartic curve: the locus of a 4th order polynomial in the plane. This quartic may be seen naturally as an elliptic curve and then parametrized by Jacobi elliptic functions. Taking $R = \rho\sin\fe, X = R\cos\nu, Y = R\sin\nu$, the quartic is the projection to the $XY$-plane of the intersection of the quadratic surfaces:
\[  R^2 = X^2 + Y^2, ~~~~ (R  + eX)^2 = p^2(1 - \kappa R^2), \]
where $p, e$ are as in the proof of prop. \ref{prop:geomcoords}.  Consequently we find the parametrization:
\begin{align*}
R &= \rho\sin\frac{\alpha}{\rho}k' \text{nd}_k w - \rho\cos\frac{\alpha}{\rho} k \text{cd}_k w & X &= \rho\sin\frac{\alpha}{\rho}k' \text{cd}_k w - \rho\cos\frac{\alpha}{\rho} k \text{nd}_k w & Y&= \rho\tan\frac{\beta}{\rho}\cos\frac{\alpha}{\rho}\text{sd}_k w. 
\end{align*}
where $k = \sin\frac{\alpha\epsilon}{\rho}, k' = \cos\frac{\alpha\epsilon}{\rho}$. Since the area is swept at a constant rate, one computes:
\[ \rho\sin\frac{\alpha}{\rho}~d\ell = \rho\sin\fe~dw.\]
 which integrates to give a 'curved Kepler equation', i.e. the relation between position and time through: \begin{equation}\label{eq:KepEq}
    \ell = \arccos \text{cd}_k w - \frac{\cot\frac{\alpha}{\rho}}{2}\log \frac{1 + k\text{sn}_k w}{1 - k\text{sn}_k w}
\end{equation}
It turns out that a geometric definition of eccentric anomaly, $u$ (see fig.~\ref{fig:eccanom}), is the Jacobi amplitude of $w$: \[ du = \text{dn}_k w~dw,\] which can be established by using spherical trigonometry to give the position in terms of $u$, and some straightforward, although tedious, simplification.

\begin{figure}[!htb]
\centering
\includegraphics[scale = .4]{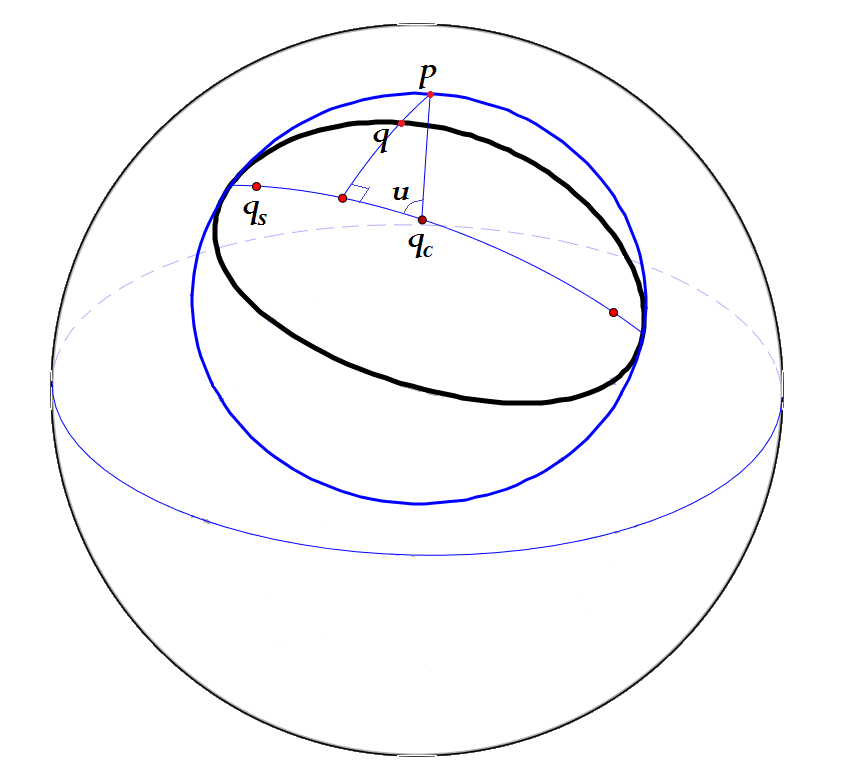}
\caption{\footnotesize{A geometric eccentric anomaly, $u$, for a (non-circular) Keplerian conic on the sphere. One inscribes a circle around the conic and to a point $q$ on the conic assigns the angle $u :=\angle(pc, q_c, p)$, where $p$ is the intersection of the circle with the perpendicular dropped from $q$ to the major axis.}}\label{fig:eccanom}
\end{figure}

\begin{rmrk}
Projecting the spherical orbits from the south pole leads as well to quartic curves in the tangent plane at $q_s$, however along these quartics the time-parametrization is no longer by sweeping area at a constant rate --as it is for orthogonal projection-- which we found only led to complicated expressions. Another parametrization of the orbits is presented in \cite{Koz} eq. (28), and used to derive a different 'curved Kepler equation' from our eq. (\ref{eq:KepEq}).
\end{rmrk}

\section{Expansion of $\langle Per\rangle$}
\label{app:Per}

We will compute an expansion in powers of our small parameter $\e$ for $\langle Per\rangle$ (prop. \ref{prop:scal}) upto $O(\e^5)$, eq. (\ref{eq:PerExp}). This expansion may be found by using the 'flat eccentric anomaly', $u_o$, of eq. (\ref{eq:flatEA}) to average the terms.  With this approach, it is necessary to make use of the following formulas allowing one to translate between the major axis and eccentricity $(a,e)$ of the centrally projected planar conic and our Delaunay coordinates:
\[ a = \frac{L^2}{m^2}\left(\frac{1}{1 - \frac{\kappa}{m^4}L^2(L^2-G^2)}\right) = \rho\e \frac{\hat L^2}{m^2}(1 + O(\e^2)),~~~~~~e^2 = (1 - \frac{\hat G^2}{\hat L^2})(1 + \frac{\e^2}{m^4}\hat L^2\hat G^2)\]
(recall we set $m = m_1m_2$). Note that by eq. (\ref{eq:flatEApos}), $r = \rho\tan\fe = \rho O(\e)$, so indeed $\fe = \frac{r}{\rho} + O(\e^3)$ is $O(\e)$.

By Taylor expansion of eq. (\ref{eq:Fred}) in $\fe$, and then exchanging $\fe$ to $\frac{r}{\rho}$, we have:
\begin{equation}\label{eq:Tayexp}
Per = \e^2\left( \frac{\hat C^2}{2} - \hat G^2 + \frac23(m_2-m_1)\hat G\sqrt{\hat C^2 - \hat G^2}  \frac{r\cos\theta}{\rho}  + ((\hat C^2 - \hat G^2)\cos^2\theta - \hat G^2)\sigma\frac{r^2}{\rho^2}\right) + O(\e^5),\end{equation}
where we set $\sigma = \frac{1 -(m_1^3 + m_2^3)}{6}$. So, to determine $\langle Per\rangle = \frac{1}{2\pi}\int_0^{2\pi} Per~d\ell$ upto $O(\e^5)$, it remains to find:
\[
\langle r\cos\theta \rangle,  ~~~~~~~ \langle r^2\cos^2\theta \rangle,  ~~~~~~~ \langle r^2 \rangle.
\]
By eq. (\ref{eq:flatEA}), $d\ell = n\sqrt{a}\frac{r}{1 + \kappa r^2} du_o = n\sqrt{a} r (1 - \kappa r^2 + O(\e^4))~du_o$, where 
\[ n\sqrt{a} = \frac{m^2}{\rho\e\hat L^2} + O(\e).\] 

Hence:
\[ \langle r\cos\theta \rangle = \frac{n\sqrt{a}}{2\pi} \int_0^{2\pi}r^2\cos\theta -  \kappa r^3\cos\theta ~du_o + O(\e^3).\]
Using $\theta = \nu + g$, and the expressions for $r, r\cos\nu, r\sin \nu$ from eq. (\ref{eq:flatEApos}), we obtain:
\[ \langle r\cos\theta \rangle = a^2 e n \sqrt{a}\cos g \left( -\frac32 + \kappa a ( 2 + \frac{e^2}{2}) \right). \]
Or, in terms of the (scaled) Delaunay coordinates:
\begin{equation}\label{eq:avg1}
\langle r\cos\theta \rangle = \rho\e \cos g~ \hat L\sqrt{\hat L^2 - \hat G^2}\left( - \frac32 + \frac{\e}{2\rho m^2}(5\hat L^2 - \hat G^2) \right) + O(\e^3).
\end{equation}
Likewise, one computes:
\begin{align}\label{eq:avg2}
\langle r^2\cos^2\theta \rangle &= \frac{\rho^2\e^2}{2m^4}\hat L^2\left( 2\hat L^2 + (3 + 5\cos 2g)(\hat L^2 - \hat G^2)\right) + O(\e^4),  & \langle r^2 \rangle&= \frac{\rho^2\e^2}{2m^4}\hat L^2(5\hat L^2 - 3\hat G^2) + O(\e^4).
\end{align}

Combining eqs. (\ref{eq:avg1}), (\ref{eq:avg2}) with the averaged eq. (\ref{eq:Tayexp}), yields:
\begin{equation}\label{eq:PerExp}
    \begin{split}
\langle Per\rangle = \e^2\left(\frac{\hat C^2}{2} - \hat G^2\right) + \e^3\left(m_\Delta\hat L\hat G\sqrt{(\hat C^2 - \hat G^2)(\hat L^2 - \hat G^2)}\cos g\right)  \\
+ \e^4\left[ \tilde m \hat L^2 \left( (\hat C^2 - \hat G^2)(2\hat L^2 + (\hat L^2 - \hat G^2)(3 + 5\cos 2g)) - \hat G^2(5\hat L^2 - 3\hat G^2)\right) \right. \\
\left.  - \frac{m_\Delta}{3\rho m^2}\hat L\hat G\sqrt{(\hat L^2 - \hat G^2)(\hat C^2 - \hat G^2)}(5\hat L^2 - \hat G^2)\cos g  \right] + O(\e^5),\\
    \end{split}
\end{equation}
where we set $m_\Delta = m_1-m_2, \tilde m = \frac{1 - m_1^3 - m_2^3}{12m^4}$, and $m = m_1m_2$.

\begin{rmrk}
When the masses are equal, one may avoid the need to take expansions as we have done here by using the simplified expression in remark \ref{rmrk:eqMass}, and making use of the different anomalies presented in section \ref{sec:anoms} to obtain explicit expressions. 
\end{rmrk}

Taking into account higher order terms of $\langle Per\rangle$ leads to finer descriptions of the orbits. For example above we have for the most part worked at order $\e^2$, at which we see the precession properties described for our continued orbits. Considering the order 3-terms (or order 4-terms with equal masses), one arrives at a more precise description of the near collision orbits. Namely, upon fixing a value of $L$, the phase portrait of the secular dynamics on the $(G,g)$ cylinder near $G= 0$ is qualitatively as that of the pendulum: there are two unstable (saddle) fixed points with $G = 0, g = 0, \pi$, having a heteroclinic connection. These fixed points lift to the red collision orbits of fig.~\ref{fig:orbs}. One corresponds to a collision orbit with body 1 in the northern hemisphere and body 2 in the southern hemisphere while the other collision orbit has body 2 in the northern hemisphere.

\begin{rmrk}\label{rmrk:split}
For the true dynamics, one expects a splitting of these stable and unstable manifolds. If one could establish a transversal intersection between these manifolds, one would obtain random motions near these collision orbits along the equator of the following form.  For any sequence $s_1,s_2,...$ with $s_k\in \{1,2\}$, there would exist an orbit for which, during the time interval $[nT, (n+1)T]$  the two bodies are closely following the periodic collision orbit with body $s_n$ in the northern hemisphere, where $T>0$ is some sufficiently long time period.
\end{rmrk}

\end{document}